\documentclass[12pt]{amsart}
\usepackage{a4wide,enumerate,xcolor}
\usepackage{amsmath,graphicx,comment}
\allowdisplaybreaks

\let\pa\partial
\let\na\nabla
\let\eps\varepsilon
\newcommand{\N}{{\mathbb N}}
\newcommand{\R}{{\mathbb R}}
\newcommand{\diver}{\operatorname{div}}

\newcommand{\simplex}{\mathcal{D}}

\newtheorem{theorem}{Theorem}
\newtheorem{lemma}[theorem]{Lemma}

\newtheorem{definition}{Definition}

\begin{document}

\title[Weak--strong uniqueness for cross-diffusion systems]{Weak--strong uniqueness for general \\ cross-diffusion systems with volume filling} 

\author[M. Heitzinger]{Maria Heitzinger}
\address{Institute of Analysis and Scientific Computing, TU Wien, Wiedner Hauptstra\ss e 8--10, 1040 Wien, Austria}
\email{maria.heitzinger@tuwien.ac.at} 

\author[A. J\"ungel]{Ansgar J\"ungel}
\address{Institute of Analysis and Scientific Computing, TU Wien, Wiedner Hauptstra\ss e 8--10, 1040 Wien, Austria}
\email{juengel@tuwien.ac.at} 

\date{\today}

\thanks{The authors acknowledge partial support from   
the Austrian Science Fund (FWF), grant 10.55776/F65, and from the Austrian Federal Ministry for Women, Science and Research and implemented by \"OAD, project MultHeFlo. This work has received funding from the European Research Council (ERC) under the European Union's Horizon 2020 research and innovation programme, ERC Advanced Grant NEUROMORPH, no.~101018153. For open-access purposes, the authors have applied a CC BY public copyright license to any author-accepted manuscript version arising from this submission.} 

\begin{abstract}
The weak--strong uniqueness of solutions to a broad class of cross-diffusion systems with volume filling is established. In general, the diffusion matrices are neither symmetric nor positive definite. This issue is overcome by supposing that the equations possess a Boltzmann-type entropy structure, which ensures the existence of bounded weak solutions. In this framework, general conditions on the mobility matrix are identified that allow for the proof of the weak--strong uniqueness property by means of the relative entropy method. The core idea consists in analyzing an augmented mobility matrix that is positive definite only on a specific subspace. Several examples that meet the required assumptions are provided, together with a discussion on possible extensions. 
\end{abstract}

% \paragraph{Keywords:}  
 
% \paragraph{AMS classification:}  

\maketitle

%%%%%%%%%%%%%%%%%%%%%%%%%%%%%%%%%%%%%%%%%%%%%%%%%%%%%%%%%%%%%%%%%%%

\section{Introduction}

Multicomponent systems in natural sciences, like gas mixtures, transport of charged particles, and population dynamics, can be described by systems of diffusion equations. When the flux of one component is induced by the gradient of the other components, the diffusion matrix is not diagonal -- a phenomenon known as cross-diffusion. In applications, the diffusion matrix is usually neither symmetric nor positive definite, but there exists an entropy structure that provides gradient estimates. We focus on cross-diffusion systems incorporating volume-filling effects, where particles or individuals have finite size, limiting the maximal achievable concentration. In such cases, many cross-diffusion models exhibit a Boltzmann entropy structure, which ensures the boundedness of solutions.

While the existence of global weak solutions can be proved under reasonable conditions by using, e.g., the boundedness-by-entropy method \cite{Jue15}, the uniqueness of weak solutions is a delicate and mainly open problem. Progress has been made so far only for systems with a special structure \cite{CaGu24,CLR17,GeJu18,ZaJu17}, restricting the size of the diffusion coefficients \cite{SeWi21}, or considering triangular systems \cite{BDD25,ChDu25}. The weak--strong uniqueness property allows for a study of more general problems. The aim is to show that a weak solution and a strong solution (if it exists) emanating from the same initial data coincide. Weak--strong uniqueness means that classical solutions are stable within the class of weak solutions. This property is usually proved by means of the relative entropy, which is a kind of ``nonlinear distance'' between the weak and strong solution. While there exist several proofs for particular systems in the literature, we present in this paper a unifying approach for weak--strong uniqueness. We show that the general proof can be applied to the models considered in the literature and to further models for which the weak--strong uniqueness property is new. 

\subsection{Setting}

The dynamics of the volume fractions $u_1,\ldots,u_n$ of a multicomponent system (gas components, population species, ionic compounds, etc.) is assumed to be governed by the cross-diffusion equations
\begin{align}\label{1.eq}
  \pa_t u_i = \diver\bigg(\sum_{j=1}^n A_{ij}(u)\na u_j\bigg)
  \quad\mbox{in }\Omega,\ t>0,\ i=1,\ldots,n,
\end{align}
with the initial and no-flux boundary conditions
\begin{align}\label{1.bic}
  u_i(0)=u_i^0\mbox{ in }\Omega, \quad
  \sum_{j=1}^n A_{ij}(u)\na u_j\cdot\nu=0\mbox{ on }\pa\Omega,\ t>0,
  \ i=1,\ldots,n,
\end{align}
where $\Omega\subset\R^d$ ($d\ge 1$) is a bounded domain, $\nu$ is the exterior unit normal vector of $\pa\Omega$, which is assumed to exist, and $u=(u_1,\ldots,u_n)$ is the solution vector. The volume fractions satisfy the volume-filling condition $\sum_{i=0}^n u_i=1$, where $u_0$ is the solvent in case of fluid mixtures or the void in case of population dynamics. Thus, we require that $u_i(t)\ge 0$ and $\sum_{i=1}^n u_i(t)\le 1$ in $\Omega$. This property is achieved by exploiting the entropy structure of equations \eqref{1.eq}. 

More precisely, we introduce the Boltzmann entropy density
\begin{align}\label{1.h}
  h(u) = \sum_{i=0}^n u_i(\log u_i-1), \quad u\in\simplex,\  u_0 = 1 - \sum_{i=1}^n u_i,
\end{align}
where $\simplex := \{u\in(0,1)^n: \sum_{i=1}^n u_i<1\}$. Then, with the so-called entropy variables $w_i=\pa h/\pa u_i=\log(u_i/u_0)$ for $i=1,\ldots,n$ and the mobility matrix $B(u):=A(u)h''(u)^{-1}\in\R^{n\times n}$, equations \eqref{1.eq} can be written equivalently as
\begin{align*}
  \pa_t u_i = \diver\bigg(\sum_{j=1}^n B_{ij}(u)\na w_j\bigg), 
  \quad i=1,\ldots,n.
\end{align*}
Solving this problem with respect to $w=(w_1,\ldots,w_n)$, the volume fractions $u_i$ are computed from $u_i=\exp w_i/(1+\sum_{j=1}^n\exp w_j)$, showing that $u(x,t)\in\simplex$ for a.e.\ $(x,t)\in\Omega_T:=\Omega\times(0,T)$ and consequently $u\in L^\infty(\Omega_T;\R^n)$. This idea was first used in \cite{BDPS10} and later generalized in \cite{Jue15}. The existence of a global bounded weak solution was proved in \cite[Theorem 2]{Jue15} under the following assumptions: $u^0\in L^1(\Omega;\R^n)$ is such that $u^0(x)\in\overline{\simplex}$ for a.e.\ $x\in\Omega$, $A\in C^0(\overline\simplex;\R^{n\times n})$, and there exist $0<s\le 1$ and $c_A>0$ such that 
\begin{align}\label{1.posdef}
  z^Th''(u)A(u)z \ge c_A\sum_{i=1}^n u_i^{2s-2}z_i^2
  \quad\mbox{for all }z\in\R^n,\ u\in\simplex.
\end{align}
The existence result also holds for $s>1$ assuming additionally that $u\mapsto A_{ij}(u)u_j^{1-s}$ is bounded, but we require $s\le 1$ in this paper.

The aim of this paper is to show that the weak--strong uniqueness property holds under rather natural assumptions.

%%%%%%%%%%%%%%

\subsection{State of the art}

The uniqueness of weak solutions to certain cross-diffusion systems was proved by using Gajewski's entropy method, combined with the $H^{-1}(\Omega)$ method \cite{ZaJu17} or with $L^2(\Omega)$ estimations \cite{GeJu18}. However, these techniques are applicable only to cross-diffusion systems possessing a particular structure that significantly simplifies the problem. A uniqueness proof for weak solutions for a thin-film solar-cell model was shown in \cite{BBEP20,SeWi21}, assuming that the diffusion coefficients are sufficiently close to one another.

To allow for more general cross-diffusion systems, we adopt a weaker uniqueness concept. An example is the weak--strong uniqueness. This property was proved, for the first time, for the incompressible Navier--Stokes equations \cite{Pro59,Ser63} and later for reaction--diffusion systems \cite{Fis17}. The idea, which goes back to \cite{Daf79}, is to estimate the so-called relative entropy $H(u|v)$, which is a kind of distance between two functions $u$ and $v$. The relative Boltzmann entropy was first used to compare two probability measures \cite{MoKo81}. Later, it has been discovered that the relative entropy can be augmented to diffusion equations \cite{Tos96}. The idea is to show that $H(u|v)$ satisfies the inequality $(dH/dt)(u|v)\le CH(u|v)$ for some $C>0$, where $u$ and $v$ are solutions to the same evolution problem. If $u$ and $v$ have the same initial data then $H(u(0)|v(0))=0$, so Gronwall's inequality gives $H(u(t)|v(t))=0$ and hence $u(t)=v(t)$ for $t>0$. 

This technique has been applied to a number of cross-diffusion systems arising from applications, for instance \cite{ChJu19,Hop22,JPZ22,LaMa23} for population dynamics models, \cite{HiJu24} for Nernst--Planck-type systems (but using a different entropy), and \cite{DeZa24} for a fractional population model. The proofs require very technical truncation procedures, which can be avoided in volume-filling models. For these problems, weak--strong uniqueness was shown in \cite{HoBu22} for a thin-film solar-cell model, in \cite{JuMa24} for an ion-transport model, and in \cite{Liu25} for a volume-filling population model. In this paper, we show that the relative entropy method can be applied to a broad class of cross-diffusion systems possessing a Boltzmann entropy structure. 

%%%%%%%%%

\subsection{Key ideas}

The task is to differentiate the relative entropy 
\begin{align*}
  H(u|v) = \sum_{i=0}^n\int_\Omega\bigg(
  u_i\log\frac{u_i}{v_i}-u_i+v_i\bigg)dx
\end{align*}
with respect to time for a weak solution $u$ and a positive strong solution $v$ and to insert equations \eqref{1.eq}. The key idea is to consider the solvent concentration $u_0$ as an independent variable and to formulate the cross-diffusion equations for the augmented solution vector $\bar{u}=(u_0,u)\in[0,1]^{n+1}$, leading to
\begin{align*}
  \pa_t u_i = \diver\bigg(\sum_{j=0}^n \bar{B}_{ij}(\bar{u})
  \na\log u_j \bigg), \quad i=0,\ldots,n,
\end{align*}
where the augmented mobility matrix is defined by
\begin{align}\label{1.barB}
  \bar{B}_{ij}(\bar{u}) = \begin{pmatrix}
  \sum_{i,j=1}^n B_{ij}(u) & -\sum_{i=1}^n B_{i1}(u) & \cdots &
  -\sum_{i=1}^n B_{in}(u) \\
  -\sum_{j=1}^n B_{1j}(u) & B_{11}(u) & \cdots & B_{1n}(u) \\
  \vdots & \vdots & & \vdots \\
  -\sum_{j=1}^n B_{nj}(u) & B_{n1}(u) & \cdots & B_{nn}(u)
  \end{pmatrix}.
\end{align}
A formal computation, which is made precise in Lemma \ref{lem.rei}, shows that
\begin{align}
  & \frac{dH}{dt}(u|v) = J_1 + J_2, \quad\mbox{where } \label{2.dHdt} \\
  & J_1 = -\sum_{i,j=0}^n\int_\Omega  
  \bar{B}_{ij}(\bar{u})\na\log\frac{u_i}{v_i}
  \cdot\na\log\frac{u_j}{v_j}dx, \nonumber \\
  & J_2 = -\sum_{i,j=0}^n\int_\Omega \bigg(\frac{\bar{B}_{ij}(\bar{u})}{u_i}
  - \frac{\bar{B}_{ij}(\bar{v})}{v_i}\bigg)u_i\na\log v_j
  \cdot\na\log\frac{u_i}{v_i}dx. \nonumber 
\end{align} 
The equality becomes an inequality if $u$ is a weak solution. The matrix $\bar{B}(\bar{u})$ has a nontrivial kernel, since the sum of the rows or columns vanishes. We show that the matrix $(\bar{B}_{ij}(\bar{u})/\sqrt{u_iu_j})_{ij}$ is positive definite only on a subspace $L(\bar{u})$, and its range turns out to be a subset of $L(\bar u)$. This enables us to use the projection $P_L(\bar u)$ onto $L(\bar u)$ on the variable $Y$ with components $Y_i:=\sqrt{u_i}\na\log(u_i/v_i)$ for $i=0,\ldots,n$. We can therefore use \eqref{1.posdef} to estimate
\begin{align*}
  J_1 \le -c_A\sum_{i=1}^n u_i^{2s-1}|(P_L(\bar{u})Y)_i|^2 dx.
\end{align*}
For the estimate of $J_2$, we split the sum into $i=0$ and $i=1,\ldots,n$ (similarly for $j$), insert the definition of $\bar{B}(\bar{u})$, and use Young's inequality (see Section \ref{sec.wsu} for details):
\begin{align}\label{1.J2}
  J_2 &= -\sum_{i,j=0}^n\int_\Omega
  \bigg(\frac{\bar{B}_{ij}(\bar{u})}{u_i} - \frac{\bar{B}_{ij}
  (\bar{v})}{v_i}\bigg)\sqrt{u_i}\na\log v_j\cdot Y_idx \\
  &\le \frac{c_A}{2}\sum_{i=0}^n\int_\Omega u_i^{2s-1}
  |(P_L(\bar{u})Y)_i|^2 dx \nonumber \\
  &\phantom{xx}+ \frac{1}{2c_A}\sum_{i,j=0}^n\int_\Omega u_i^{2-2s}
  \bigg|\frac{\bar{B}_{ij}(\bar{u})}{u_i}
  - \frac{\bar{B}_{ij}(\bar{v})}{v_i}\bigg|^2|\na\log v_j|^2 dx.
  \nonumber 
\end{align}
The first term on the right-hand side can be absorbed by $J_1$. The second term can be estimated if $s\le 1$ (then $u_i^{2-s2}\le 1$), if $|\na\log v_j|$ is a bounded function (here, we need that $v$ is a positive strong solution in the sense of Definition \ref{def.strong} below), and if $\bar{u}\mapsto \bar{B}(\bar{u})/u_i$ is Lipschitz continuous. Under these conditions, it follows from \eqref{2.dHdt} that
\begin{align*}
  \frac{dH}{dt}(u|v) \le C\sum_{i=0}^n\int_\Omega|u_i-v_i|^2 dx
\end{align*}
for some constant $C>0$ depending on the $L^\infty$ norm of $\na\log v_j$. Applying the elementary inequality \cite[Lemma 16]{HJT22}
\begin{align}\label{1.HL2}
  u_i\log\frac{u_i}{v_i} - u_i + v_i 
  \ge \frac12\frac{|u_i-v_i|^2}{\max\{u_i,v_i\}},
\end{align}
the boundedness of $u_i$ and $v_i$ yields $(dH/dt)(u|v) \le CH(u|v)$. Since $u(0)=v(0)$, Gronwall's lemma gives the desired result $u(t)=v(t)$ for $t>0$. 

The idea to augment the mobility matrix was already used in \cite{GeJu18}. The novelty of this paper is the generalization of this idea and the identification of the conditions allowing us to establish the weak--strong uniqueness property for general cross-diffusion systems.

%%%%%%%%%%%%%%%

\subsection{Main result} 

We introduce the spaces
\begin{align}\label{1.simplex}
  \simplex = \bigg\{u\in(0,1)^n:\sum_{i=1}^n u_i<1\bigg\}, \quad
  \simplex_0 = \bigg\{\bar{u}=(u_0,u):u\in\simplex, \,
  u_0=1-\sum_{i=1}^n u_i\bigg\}
\end{align}
and impose the following hypotheses.

\begin{itemize}
\item[(H1)] Domain: $\Omega\subset\R^d$ with $d\ge 1$ is a bounded domain with Lipschitz boundary, $T>0$, and $\Omega_T:=\Omega\times(0,T)$. 
\item[(H2)] Initial data: $u^0=(u_1^0,\ldots,u_n^0)\in L^1(\Omega;\R^n)$ is such that $u^0(x)\in\overline{\simplex}$ for a.e.\ $x\in\Omega$. 
\item[(H3)] Positive definiteness: $A\in C^0(\overline{\simplex};\R^{n\times n})$ and there exist $0<s\le 1$ and $c_A>0$ such that
\begin{align*}
  z^Th''(u)A(u)z\ge c_A\sum_{i=1}^n u_i^{2s-2}z_i^2
  \quad\mbox{for all }z\in\R^n,\,u\in\simplex.
\end{align*}
\item[(H4)] Boundedness: Either
\begin{align*}
  &{\rm (i)}\phantom{i}
  \quad u\mapsto (h''(u)A(u))_{ij}u_i^{1-s}u_j^{1-s}
  \quad\mbox{for }u\in\overline\simplex,\ i,j=1,\ldots,n,
  \mbox{ or}\quad \\
  &{\rm (ii)}\quad \bar{u}\mapsto
  \frac{\bar{B}_{ij}(\bar{u})}{u_i^su_j^s}\quad\mbox{for } 
  \bar{u}\in\overline{\simplex}_0,\ i,j=0,\ldots,n,
\end{align*}
are bounded functions.
\item[(H5)] Lipschitz continuity: There exists $C_B>0$ such that for all $i,j=0,\ldots,n$ and $\bar{u}$, $\bar{v}\in\overline{\mathcal{D}}_0$,
\begin{align*}
  \bigg|\frac{\bar{B}_{ij}(\bar{u})}{u_i}
  - \frac{\bar{B}_{ij}(\bar{v})}{v_i}\bigg|
  \le C_B\sum_{k=0}^n|u_k-v_k|.
\end{align*}
\end{itemize}

Hypotheses (H1)--(H3) are needed to guarantee the existence of a global weak solution; see \cite{Jue15}. The condition $s\le 1$ is required to estimate the last term in \eqref{1.J2} (to ensure that the factor $u_i^{2-2s}$ can be bounded). This assumption is not needed in the existence analysis. Hypothesis (H4) is used to ensure the integrability of the entropy production; see \eqref{1.int} below. This property is needed to prove the entropy inequality for weak solutions, which is an ingredient of the proof of the relative entropy inequality. Finally, Hypothesis (H5) is required to estimate the integral $J_2$; see \eqref{1.J2}. We can replace this assumption by the following weaker condition:
\begin{itemize}
\item[(H5)'] There exist $C_B>0$ and $\gamma\in(0,1)$ such that for all $i,j=0,\ldots,n$ and $\bar{u}$, $\bar{v}\in\overline{\mathcal{D}}_0$,
\begin{align*}
  \bigg|\frac{\bar{B}_{ij}(\bar{u})}{u_i}
  - \frac{\bar{B}_{ij}(\bar{v})}{v_i}\bigg|
  \le C_B\sum_{k=0}^n|u_k^\gamma-v_k^\gamma|.
\end{align*}
\end{itemize}
Indeed, if Hypothesis (H5) holds, we infer from the inequality $|u_k-v_k|\le C(\gamma)|u_k^\gamma-v_k^\gamma|$ for $u_k$, $v_k\in[0,1]$ that Hypothesis (H5)' is fulfilled. A simple example fulfilling Hypothesis (H5)' but not (H5) is presented in Section \ref{sec.scalar}.

We specify the definition of weak and strong solutions. 

\begin{definition}[Weak solution]\label{def.weak}
Let $T>0$. We call $u$ a {\em weak solution} to \eqref{1.eq}--\eqref{1.bic} if 
$u_i(x,t)\in\overline{\simplex}$ for $x\in\Omega_T$,
\begin{align*}
  u_i^s\in L^2(0,T;H^1(\Omega)), \quad
  \pa_t u_i\in L^2(0,T;H^1(\Omega)'),
\end{align*}
and it holds for all $\phi_i\in L^2(0,T;H^1(\Omega))$ and $i=1,\ldots,n$ that
\begin{align*}
  \int_0^T\langle\pa_t u_i,\phi_i\rangle dt 
  + \int_0^T\int_\Omega A_{ij}(u)\na u_j\cdot\na\phi_i dxdt = 0,
\end{align*}
where $\langle\cdot,\cdot\rangle$ denotes the dual product between $H^1(\Omega)'$ and $H^1(\Omega)$. The initial datum is satisfied in the sense of $L^2(\Omega)$. 
\end{definition}

Notice that if $u$ is a weak solution to \eqref{1.eq}--\eqref{1.bic} then $u_i\in L^\infty(\Omega_T)$ and $u_i^s\in L^2(0,T;H^1(\Omega))$  imply that $u_i\in L^2(0,T;H^1(\Omega))$, since $\na u_i=s^{-1}u_i^{1-s}\na u_i^s\in L^2(\Omega_T)$. According to \cite[Theorem 2]{Jue15}, under Hypotheses (H1)--(H3), there exists a global weak solution $u$ to \eqref{1.eq}--\eqref{1.bic} in the sense of Definition \ref{def.weak}. We show in Appendix \ref{sec.app} that under Hypothesis (H4), the solution constructed in \cite{Jue15} satisfies the entropy inequality
\begin{align}\label{1.ei}
  \int_\Omega h(u(t))dx + \sum_{i,j=0}^n \int_0^t\int_\Omega
  \bar{B}_{ij}(\bar{u})\na\log u_i\cdot\na\log u_j dxds 
  \le \int_\Omega h(u^0)dx.
\end{align}
Notice that by Hypothesis (H4);
\begin{align}\label{1.int}
  \bar{B}_{ij}(\bar{u})\na\log u_i\cdot\na\log u_j
  = \frac{1}{s^2}\frac{\bar{B}_{ij}(\bar{u})}{u_i^su_j^s}
  \na u_i^s\cdot\na u_j^s, \quad i,j=0,\ldots,n,
\end{align}
is integrable since $\na u_i^s\in L^2(\Omega_T)$. 

\begin{definition}[Strong solution]\label{def.strong}
Let $T>0$. We say that $v=(v_1,\ldots,v_n)$ is a {\em positive strong solution} to \eqref{1.eq}--\eqref{1.bic} if it is a weak solution, there exists $q>0$ such that 
\begin{align*}
  v_i\ge q>0\quad\mbox{in }\Omega_T, \quad
  v_i\in L^\infty(0,T;W^{1,\infty}(\Omega))
\end{align*}
for all $i=1,\ldots,n$, and $v$ does not have anomalous diffusion, i.e., it fulfills the entropy equality
\begin{align}\label{1.ee}
  \int_\Omega h(v(t))dx + \sum_{i,j=0}^n\int_0^t\int_\Omega
  \bar{B}_{ij}(\bar{v})\na\log v_i\cdot\na\log v_j dxds
  = \int_\Omega h(v(0))dx.
\end{align}
\end{definition}

Our main result is as follows.

\begin{theorem}[Weak--strong uniqueness]\label{thm.wsu}
Let Hypotheses (H1)--(H5) hold, let $u$ be a weak solution to \eqref{1.eq}--\eqref{1.bic} satisfying \eqref{1.ei} and let $v$ be a positive strong solution. If the initial data coincide then $u(t)=v(t)$ in $\Omega$ for $t>0$.
\end{theorem}

Our result can be applied to several cross-diffusion systems arising from physical and biological applications. We summarize the models in the following table (see Section \ref{sec.exam} for details):

\bigskip
\renewcommand{\arraystretch}{1.1}
\begin{tabular}{|l|l|l|}
\hline 
Model & Weak--strong uniqueness & Reference \\ \hline
Multiphase model & new result & \\ 
Tumor-growth model & new result & \\
Modified Busenberg--Travis model & new result & \\
Maxwell-Stefan model & known & \cite{HJT22} \\
Thin-film solar-cell model & known (but different proof) 
& \cite{HoBu22} \\ 
Ion-transport model & known & \cite{JuMa24} \\
\hline
\end{tabular}

\bigskip
Our approach allows for degenerate problems (in the sense that $A_{ii}(u)=0$ if $u_i=0$) and can be extended to more general settings; see Section \ref{sec.ext}. However, it also has some limitations. For instance, we cannot consider the population model \cite[Sec.~4.2]{Jue16}
\begin{align*}
  \pa_t u_i = \diver(d_i(u_0\na u_i-u_i\na u_0)), \quad i=1,\ldots,n,
\end{align*}
where $d_i>0$, because of the non-standard degeneracy involving $u_0$. Indeed, for this example, Hypothesis (H3) changes to $z^Th''(u)A(u)z\ge c_A\sum_{i=1}^n u_0u_i^{-1}z_i^2$ \cite[(23)]{GeJu18}. The factor $u_0$ needs to be incorporated in the term $J_1$ (see \eqref{2.dHdt}), but then we need to estimate
\begin{align*}
  \frac{1}{u_0}\bigg|\frac{\bar{B}_{ij}(\bar{u})}{u_i}
  - \frac{\bar{B}_{ij}(\bar{v})}{v_i}\bigg|^2,
\end{align*}
arising from Young's inequality applied to the term $J_2$ (see \eqref{1.J2}). Unfortunately, the factor $1/u_0$ does not cancel out and we cannot expect a positive lower bound for $u_0$. Thus, non-standard degeneracies in $u_0$ are not permitted. In some sense, only cross-diffusion systems can be considered whose main part contains $u_i^\alpha\na u_i$ with $\alpha\ge 0$. 

The presented technique can be applied in principle to cross-diffusion systems without volume-filling property. Since solutions to such systems are generally not bounded, inequality \eqref{1.HL2} cannot be used anymore. This issue was overcome in \cite{Fis17} by truncating the part $-u_i\log v_i$ in the relative entropy. This truncation was successfully applied to a population cross-diffusion model in \cite{ChJu19}. However, the truncation procedure is highly technical, while a key advantage of our approach is that it avoids any truncation, leading to a surprisingly straightforward proof, which is possibly accessible to generalizations.

The paper is organized as follows. We present some results from matrix analysis in Section \ref{sec.pre}. Theorem \ref{thm.wsu} is proved in Section \ref{sec.wsu} using the ideas explained before. The result is applied to several models in Section \ref{sec.exam}. In Section \ref{sec.ext}, we present some extensions of the approach, including an example for which Hypotheses (H4) and (H5) are not satisfied. Finally, the entropy inequality for weak solutions is proved in Appendix \ref{sec.app}.

%%%%%%%%%%%%%%%%%%%%%%%%%%%%%%%%%%%%%%%%%%%%%%%%%%%%%%%%%%%%%%%

\section{Preliminaries}\label{sec.pre}

Recall definition \eqref{1.simplex} for $\simplex$ and $\simplex_0$ and let $\bar{u}\in\simplex_0$. We introduce the matrix $G(\bar{u})\in\R^{(n+1)\times(n+1)}$ with entries
$G_{ij}(\bar{u}) = \bar{B}(\bar{u})/\sqrt{u_iu_j}$ for $i,j=0,\ldots,n$, where $\bar{B}(\bar{u})$ is defined in \eqref{1.barB}. Then, by the construction of $\bar{B}_{ij}$,
\begin{align}\label{2.kerG}
  \sum_{j=0}^n G_{ij}(\bar{u})\sqrt{u_j}
  = \frac{1}{\sqrt{u_i}}\sum_{j=0}^n\bar{B}_{ij}(\bar{u}) = 0, \quad
  \sum_{i=0}^n G_{ji}(\bar{u})\sqrt{u_i}
  = \frac{1}{\sqrt{u_j}}\sum_{i=0}^n\bar{B}_{ji}(\bar{u}) = 0,
\end{align}
which implies that $\operatorname{span}\sqrt{\bar{u}}\subset \operatorname{ker}G(\bar{u})$ and $\operatorname{span}\sqrt{\bar u}\subset \operatorname{ker}G(\bar{u})^T$, where $\sqrt{\bar u}=(\sqrt{u_0},\ldots,$ $\sqrt{u_n})$. We set
\begin{align*}
  L(\bar{u}) := \{Y\in\R^{n+1}: \sqrt{\bar u}\cdot Y=0\}, \quad
  L(\bar{u})^\perp = \operatorname{span}\sqrt{\bar u},
\end{align*}
and introduce the projections
\begin{align*}
  (P_L(\bar{u})Y)_i = Y_i - \sqrt{u_i}(\sqrt{\bar u}\cdot Y), \quad
  (P_{L^\perp}(\bar{u})Y)_i = \sqrt{u_i}(\sqrt{\bar u}\cdot Y), \quad
  i=0,\ldots,n.
\end{align*}
Then \eqref{2.kerG} implies that 
\begin{align}\label{2.GPL}
  P_L(\bar{u})G(\bar{u}) = G(\bar{u})P_L(\bar{u}) = G(\bar u).
\end{align}
We show that the matrix $G(\bar{u})$ is positive definite on the subspace $L(\bar{u})$.

\begin{lemma}\label{lem.G}
Let $z\in L(\bar{u})$. Then
\begin{align*}
  z^TG(\bar{u})z \ge c_A\sum_{i=1}^n u_i^{2s-1}z_i^2.
\end{align*}
\end{lemma}

\begin{proof}
We decompose the sum over $i,j=0,\ldots,n$ into the four parts ($i,j=1,\ldots,n$), ($i=0$, $j=1,\ldots,n$), ($i=1,\ldots,n$, $j=0$), and ($i=j=0$) and use definition \eqref{1.barB} of $\bar{B}(\bar{u})$:
\begin{align*}
  z^TG(\bar{u})z &= \sum_{i,j=1}^n\frac{z_iz_j}{\sqrt{u_iu_j}}
  B_{ij}(u) - \sum_{i,j=1}^n\frac{z_iz_0}{\sqrt{u_iu_0}}B_{ij}(u)
  - \sum_{i,j=1}^n\frac{z_0z_j}{\sqrt{u_0u_j}}B_{ij}(u)
  + \frac{z_0^2}{u_0}\sum_{i,j=1}^n B_{ij}(u) \\
  &= \sum_{i,j=1}^n\bigg(\frac{z_0}{\sqrt{u_0}}-\frac{z_i}{\sqrt{u_i}}
  \bigg)B_{ij}(u)\bigg(\frac{z_0}{\sqrt{u_0}}-\frac{z_j}{\sqrt{u_j}}
  \bigg).
\end{align*}
Set $\xi_i=z_0/\sqrt{u_0}-z_i/\sqrt{u_i}$ for $i=1,\ldots,n$. Since the entries of $h''(u)^{-1}$ are given by $\delta_{ij}u_i-u_iu_j$, we find that
\begin{align*}
  (h(u)^{-1}\xi)_i &= u_i\xi_i - u_i\sum_{j=1}^n u_j\xi_j 
  = \frac{u_iz_0}{\sqrt{u_0}}-\sqrt{u_i}z_i
  - u_i\sum_{j=1}^n u_j\frac{z_0}{\sqrt{u_0}} 
  + \sum_{j=1}^n\sqrt{u_j}z_j \\
  &= \frac{u_iz_0}{\sqrt{u_0}}-\sqrt{u_i}z_i
  - u_i(1-u_0)\frac{z_0}{\sqrt{u_0}} - \sqrt{u_0}z_0
  = -\sqrt{u_i}z_i =: \eta_i,
\end{align*}
where we used the properties $\sum_{j=1}^n u_j=1-u_0$ and $z\in L(\bar{u})$, i.e.\ $\sum_{j=1}^n\sqrt{u_j}z_j = -\sqrt{u_0}z_0$. This yields $\xi=h''(u)\eta$ and, by Hypothesis (H3),
\begin{align*}
  z^TG(\bar{u})z &= \xi^T B(u)\xi 
  = \eta^T h''(u)^T(A(u)h''(u)^{-1})h''(u)\eta \\
  &= \eta^T h''(u)A(u)\eta \ge c_A\sum_{i=1}^n u_i^{2s-2}\eta_i^2
  = c_A\sum_{i=1}^n u_i^{2s-1}z_i^2,
\end{align*}
which finishes the proof.
\end{proof}

%%%%%%%%%%%%%%%%%%%%%%%%%%%%%%%%%%%%%%%%%%%%%%%%%%%%%%%%%%%%%%%

\section{Proof of Theorem \ref{thm.wsu}}\label{sec.wsu}

We first derive the relative entropy inequality.

\begin{lemma}[Relative entropy inequality]\label{lem.rei}
Let $u$ be a weak solution satisfying \eqref{1.ei} and let $v$ be a positive strong solution to \eqref{1.eq}--\eqref{1.bic}. Then for any $0<t<T$,
\begin{align}\label{3.rei}
  H(u(t)|v(t)) &- H(u(0)|v(0)) 
  \le -\sum_{i=0}^n\int_0^t\int_\Omega\bar{B}_{ij}(\bar{u})
  \na\log\frac{u_j}{v_j}\cdot\na\log\frac{u_i}{v_i}dxds, \\
  &-\sum_{i=0}^n\int_0^t\int_\Omega\bigg(
  \frac{\bar{B}_{ij}(\bar{u})}{u_i}-\frac{\bar{B}_{ij}(\bar{v})}{v_i}
  \bigg)u_i\na\log v_j\cdot\na\log\frac{u_i}{v_i}dxds. \nonumber 
\end{align}
\end{lemma}

\begin{proof}
The difference of the relative entropies becomes
\begin{align}\label{3.aux}
  H(u(t)&|v(t)) - H(u(0)|v(0))
  = \int_0^t\frac{d}{dt}\int_\Omega
  \bigg(h(u)-h(v)-\sum_{i=0}^n\log v_i(u_i-v_i)\bigg)dxds \\
  &= \int_\Omega\big(h(u(t)) - h(u(0))\big)dx
  - \int_\Omega\big(h(v(t)) - h(v(0))\big)dx \nonumber \\
  &\phantom{xx}- \sum_{i=0}^n\int_0^t\frac{d}{dt}\int_\Omega
  (u_i-v_i)\log v_i dxds, \nonumber 
\end{align}
recalling definition \eqref{1.h} of $h$. Using the chain rule \cite[Sec.~2.1.1]{GaMa96}, the last term can be reformulated as
\begin{align*}
  -\sum_{i=0}^n&\int_0^t\frac{d}{dt}\int_\Omega
  (u_i-v_i)\log v_i dxds \\
  &= -\sum_{i=0}^n\int_0^t\langle\pa_t\log v_i,u_i-v_i\rangle dxds
  - \sum_{i=0}^n\int_0^t\langle \pa_t(u_i-v_i),\log v_i\rangle dxds \\
  &= -\sum_{i=0}^n\int_0^t\bigg\langle \pa_t v_i,\frac{u_i}{v_i}-1
  \bigg\rangle ds 
  - \sum_{i=0}^n\int_0^t\langle \pa_t(u_i-v_i),\log v_i\rangle dxds \\
  &= \sum_{i,j=0}^n\int_0^t\int_\Omega\bar{B}_{ij}(\bar{v})
  \na\log v_j\cdot\na\frac{u_i}{v_i}dxds \\
  &\phantom{xx}+ \sum_{i,j=0}^n\int_0^t\int_\Omega
  \big(\bar{B}_{ij}\big(\bar{u})\na\log u_j
  - \bar{B}_{ij}(\bar{v})\na\log v_j\big)\cdot\na\log v_i dxds.
\end{align*}
The term involving $\log u_j$ is interpreted as
\begin{align*}
  \bar{B}_{ij}(\bar{u})\na\log u_j
  = \frac{u_i^s}{s}\frac{\bar{B}_{ij}(\bar{u})}{u_i^su_j^s}\na u_j^s,
\end{align*}
which by Hypothesis (H4) is a function in $L^2(\Omega_T)$. We insert the entropy inequality \eqref{1.ei} for $u_i$ and the entropy equality \eqref{1.ee} for $v_i$ in the first two terms of \eqref{3.aux}. Then some terms cancel, and we obtain 
\begin{align*}
  H(u(t)|v(t)) - H(u(0)|v(0))
  &\le -\sum_{i,j=0}^n\int_0^t\int_\Omega\bar{B}_{ij}(\bar{u})
  \na\log u_j\cdot\na\log\frac{u_i}{v_i} dxds \\
  &\phantom{xx}+ \sum_{i,j=0}^n\int_0^t\int_\Omega\bar{B}_{ij}(\bar{v})
  \na\log v_j\cdot\na\frac{u_i}{v_i}dxds.
\end{align*}
We add and subtract the sum of $\bar{B}_{ij}(\bar{u})\na\log v_j\cdot\na\log(u_i/v_i)$, which shows \eqref{3.rei}.
\end{proof}

We proceed with the proof of Theorem \ref{thm.wsu}. For this, we estimate the two integrals on the right-hand side of \eqref{3.rei}:
\begin{equation}\label{3.I12}
\begin{aligned}
  I_1 &= -\sum_{i,j=0}^n\int_0^t\int_\Omega\bar{B}_{ij}(\bar{u})
  \na\log\frac{u_j}{v_j}\cdot\na\log\frac{u_i}{v_i}dxds, \\
  I_2 &= -\sum_{i,j=0}^n\int_0^t\int_\Omega\bigg(
  \frac{\bar{B}_{ij}(\bar{u})}{u_i}-\frac{\bar{B}_{ij}(\bar{v})}{v_i}
  \bigg)u_i\na\log v_j\cdot\na\log\frac{u_i}{v_i}dxds.
\end{aligned}
\end{equation}
We start with the estimation of $I_1$. Set
\begin{align*}%\label{3.Y}
  Y_i := \sqrt{u_i}\na\log\frac{u_i}{v_i}, \quad i=0,\ldots,n.
\end{align*}
This function may be not integrable if $s<1/2$, but the functions $u_i^{2s-1}|\sqrt{u_i}\log u_i|^2 = s^{-2}|\na u_i^s|^2$ and hence $u_i^{2s-1}|Y_i|^2$, as well as $G_{ij}(\bar{u})Y_iY_j = \bar{B}_{ij}(\bar{u})\na\log u_j\cdot\na\log u_i$ are integrable. It follows from \eqref{2.GPL} and Lemma \ref{lem.G} that 
\begin{align}\label{3.I1}
  I_1 &= -\sum_{i,j=0}^n\int_0^t\int_\Omega G_{ij}(\bar{u})Y_iY_j dxds
  = -\sum_{i,j=0}^n\int_0^t\int_\Omega
  G_{ij}(\bar{u})(P_L(\bar{u})Y)_i(P_L(\bar{u})Y)_j dxds \\
  &\le -c_A\sum_{i=1}^n\int_0^t\int_\Omega u_i^{2s-1}
  |(P_L(\bar{u})Y)_i|^2 dxds. \nonumber 
\end{align}
This term will be used to absorb part of $I_2$. We split $Y_i=(P_L(\bar{u})Y)_i +(P_{L^\perp}(\bar{u})Y)_i$ to find that
\begin{align*}
  & I_2 = I_{21} + I_{22}, \quad\mbox{where} \\
  & I_{21} = -\sum_{i,j=0}^n\int_0^t\int_\Omega\bigg(
  \frac{\bar{B}_{ij}(\bar{u})}{u_i}-\frac{\bar{B}_{ij}(\bar{v})}{v_i}
  \bigg)\sqrt{u_i}\na\log v_j\cdot (P_L(\bar{u})Y)_i dxds, \\
  & I_{22} = -\sum_{i,j=0}^n\int_0^t\int_\Omega\bigg(
  \frac{\bar{B}_{ij}(\bar{u})}{u_i}-\frac{\bar{B}_{ij}(\bar{v})}{v_i}
  \bigg)\sqrt{u_i}\na\log v_j\cdot (P_{L^\perp}(\bar{u})Y)_i dxds.
\end{align*}
For the estimate of $I_{21}$, we will use Hypothesis (H5). The idea is to obtain one part involving the differences $|u_i-v_i|^2$ and another part that can be absorbed by $I_1$. To this end, we split the sum into the four parts ($i,j=1,\ldots,n$), ($i=1,\ldots,n$, $j=0$), ($i=0$, $j=1,\ldots,n$), and ($i=j=0$):
\begin{align*}
  & I_{21} = I_{211} + I_{212} + I_{213} + I_{214}, \quad\mbox{where} \\
  & I_{211} = -\sum_{i,j=1}^n\int_0^t\int_\Omega\bigg(
  \frac{\bar{B}_{ij}(\bar{u})}{u_i}-\frac{\bar{B}_{ij}(\bar{v})}{v_i}
  \bigg)\sqrt{u_i}\na\log v_j\cdot (P_L(\bar{u})Y)_i dxds, \\
  & I_{212} = -\sum_{i=1}^n\int_0^t\int_\Omega\bigg(
  \frac{\bar{B}_{i0}(\bar{u})}{u_i}-\frac{\bar{B}_{i0}(\bar{v})}{v_i}
  \bigg)\sqrt{u_i}\na\log v_0\cdot (P_L(\bar{u})Y)_i dxds, \\
  & I_{213} = -\sum_{j=1}^n\int_0^t\int_\Omega\bigg(
  \frac{\bar{B}_{0j}(\bar{u})}{u_0}-\frac{\bar{B}_{0j}(\bar{v})}{v_0}
  \bigg)\sqrt{u_0}\na\log v_j\cdot (P_L(\bar{u})Y)_0 dxds, \\
  & I_{214} = -\int_0^t\int_\Omega\bigg(
  \frac{\bar{B}_{00}(\bar{u})}{u_0}-\frac{\bar{B}_{00}(\bar{v})}{v_0}
  \bigg)\sqrt{u_0}\na\log v_0\cdot (P_L(\bar{u})Y)_0 dxds.
\end{align*}

We estimate $I_{211},\ldots,I_{214}$ term by term. By Young's inequality for any $\eps>0$,
\begin{align*}
  I_{211} &\le \frac{\eps}{4}\sum_{i=1}^n\int_0^t\int_\Omega u_i^{2s-1}
  |(P_L(\bar{u})Y)_i|^2 dxds \\
  &\phantom{xx}+ \frac{1}{\eps}\sum_{i,j=1}^n\int_0^t\int_\Omega
  \bigg|\frac{\bar{B}_{ij}(\bar{u})}{u_i}
  - \frac{\bar{B}_{ij}(\bar{v})}{v_i}\bigg|^2
  |\na\log v_j|^2 u_i^{2-2s}dxds \\
  &\le \frac{\eps}{4}\sum_{i=1}^n\int_0^t\int_\Omega u_i^{2s-1}
  |(P_L(\bar{u})Y)_i|^2 dxds 
  + C(\eps)\sum_{i=0}^n\int_0^t\int_\Omega|u_i-v_i|^2 dxds,
\end{align*}
where the constant $C(\eps)>0$ depends on the $L^\infty(\Omega_T)$ norm of $\na\log v_j$ and we have used in the last step Hypothesis (H5) as well as the property $u_i^{2-2s}\le 1$ (which holds because of $u_i\le 1$ and $s\le 1$). By similar arguments as before,
\begin{align*}
  I_{212} &\le \frac{\eps}{4}\sum_{i=1}^n\int_0^t\int_\Omega u_i^{2s-1}
  |(P_L(\bar{u})Y)_i|^2 dxds 
  + C(\eps)\sum_{i=0}^n\int_0^t\int_\Omega|u_i-v_i|^2 dxds,
\end{align*}
where $C(\eps)>0$ depends on the $L^\infty(\Omega_T)$ norm of $\na\log v_0$. It follows from $P_L(\bar{u})Y\in L(\bar{u})$ that $\sqrt{u_0}(P_L(\bar{u})Y)_0 = -\sum_{k=1}^n\sqrt{u_k}(P_L(\bar{u})Y)_k$. This leads to
\begin{align*}
  I_{213} &= \sum_{j,k=1}^n\int_0^t\int_\Omega
  \bigg(\frac{\bar{B}_{0j}(\bar{u})}{u_0}
  - \frac{\bar{B}_{0j}(\bar{v})}{v_0}\bigg)
  \na\log v_j\cdot\big(\sqrt{u_k}(P_L(\bar{u})Y)_k\big)dxds \\
  &\le \frac{\eps}{4}\sum_{k=1}^n\int_0^t\int_\Omega u_k^{2s-1}
  |(P_L(\bar{u})Y)_k|^2 dxds 
  + C(\eps)\sum_{i=0}^n\int_0^t\int_\Omega|u_i-v_i|^2 dxds.
\end{align*}
Similarly, the last term $I_{214}$ is estimated according to
\begin{align*}
  I_{214} &\le \frac{\eps}{4}\sum_{k=1}^n\int_0^t\int_\Omega u_k^{2s-1}
  |(P_L(\bar{u})Y)_k|^2 dxds 
  + C(\eps)\sum_{i=0}^n\int_0^t\int_\Omega|u_i-v_i|^2 dxds.
\end{align*}
Summarizing, we find that
\begin{align}\label{3.I21}
  I_{21} \le \eps\sum_{i=1}^n\int_0^t\int_\Omega u_i^{2s-1}
  |(P_L(\bar{u})Y)_i|^2 dxds 
  + C(\eps)\sum_{i=0}^n\int_0^t\int_\Omega|u_i-v_i|^2 dxds.
\end{align}

To estimate $I_{22}$, we use the definition of $P_{L^\perp}(\bar{u})$ and the property $\sum_{k=0}^n\na v_k=0$:
\begin{align*}
  (P_{L^\perp}Y)_i = -\sqrt{u_i}\sum_{k=1}^n u_k\na\log v_k
  = -\sqrt{u_i}\sum_{k=1}^n(u_k-v_k)\na\log v_k.
\end{align*}
Then the Lipschitz continuity of Hypothesis (H5) gives
\begin{align}\label{3.I22}
  I_{22} &= \sum_{i,j=0}^n\sum_{k=1}^n\int_0^t\int_\Omega
  \bigg(\frac{\bar{B}_{ij}(\bar{u})}{u_0}
  - \frac{\bar{B}_{ij}(\bar{v})}{v_0}\bigg)
  u_i(u_k-v_k)\na\log v_j\cdot\na\log v_k dxds \\
  &\le C\sum_{i=0}^n\int_0^t\int_\Omega|u_i-v_i|^2 dxds. \nonumber 
\end{align}

To conclude the proof, we choose $0<\eps\le c_A$. We deduce from estimates \eqref{3.I1} for $I_1$ and \eqref{3.I21}--\eqref{3.I22} for $I_2=I_{21}+I_{22}$ that
\begin{align*}
   H(u(t)|v(t)) - H(u(0)|v(0))  = I_1 + I_2 
  \le C(\eps) \sum_{i=0}^n\int_0^t\int_\Omega|u_i-v_i|^2 dxds.
\end{align*}
It follows from \cite[Lemma 16]{HJT22} for $0\le y\le 1$, $0<z\le 1$ that
\begin{align*}
  y\log\frac{y}{z}-y+z \ge \frac12\frac{(y-z)^2}{\max\{y,z\}}
  \ge \frac12(y-z)^2
\end{align*}
and consequently,
\begin{align*}%\label{3.L2}
  H(u|v) \ge \frac12\sum_{i=0}^n\int_\Omega|u_i-v_i|^2dx.
\end{align*}
This gives for a.e.\ $t>0$:
\begin{align*}
  H(u(t)|v(t)) - H(u(0)|v(0)) \le 2C(\eps)\int_0^t H(u|v)ds.
\end{align*}
Since $u$ and $v$ have the same initial data, Gronwall's lemma shows that $H(u(t)|v(t))=0$ and hence $u(t)=v(t)$ in $\Omega$ for $t>0$. The proof is finished.

%%%%%%%%%%%%%%%%%%%%%%%%%%%%%%%%%%%%%%%%%%%%%%%%%%%%%%%%%%%%%%%

\section{Examples}\label{sec.exam}

We present some examples for which Hypotheses (H3)--(H5) are satisfied. In all examples, we consider the equations in a bounded domain $\Omega\subset\R^d$, impose no-flux boundary conditions, and assume that the problem-independent Hypotheses (H1)--(H2) hold. 

\subsection{Scalar diffusion model}\label{sec.scalar}

For illustration, we start with the case $n=1$. We consider the scalar diffusion equation
\begin{align*}
  \pa_t u_1 = \diver(u_1^\alpha(1-u_1)\na u_1)
\end{align*}
with initial and no-flux boundary conditions and $0\le\alpha\le 1$. Standard arguments show that this problem admits a global weak solution satisfying $0\le u_1\le 1$ in $\Omega_T$. The uniqueness of weak solutions can be shown by means of the $H^{-1}(\Omega)$ method, but we use this example to illustrate our technique. We set $u_0=1-u_1$ and introduce the entropy density
\begin{align*}
  h(u_1) = u_1(\log u_1-1) + (1-u_1)(\log(1-u_1)-1) 
  = u_1\log\frac{u_1}{u_0} + \log u_0.
\end{align*}
The $1\times 1$ diffusion matrix equals $A(u_1)=u_0 u_1^\alpha$, and the Hessian becomes $h''(u_1)=1/(u_0u_1)$. Then $h''(u_1)A(u_1)=u_1^{\alpha-1}$, and Hypotheses (H3) and (H4) are satisfied with $s=(\alpha+1)/2$. Here, we need $\alpha\le 1$ to ensure that $s\le 1$. The $1\times 1$ mobility matrix becomes $B(u_1)=A(u_1)h''(u_1)^{-1}=u_0^2u_1^{\alpha+1}$, which yields the augmented matrix
\begin{align*}
  \bar{B}(\bar{u}) = u_0^2u_1^{\alpha+1}\begin{pmatrix}
  \phantom{x}1 & -1 \\ -1 & \phantom{x}1 \end{pmatrix}.
\end{align*}
We compute for $j=0,1$:
\begin{align*}
  \bigg|\frac{\bar{B}_{0j}(\bar{u})}{u_0}
  - \frac{\bar{B}_{0j}(\bar{v})}{v_0}\bigg|
  &= |u_0 u_1^{\alpha+1} - v_0 v_1^{\alpha+1}|
  \le C(|u_0-v_0| + |u_1-v_1|), \\
  \bigg|\frac{\bar{B}_{1j}(\bar{u})}{u_1}
  - \frac{\bar{B}_{1j}(\bar{v})}{v_1}\bigg|
  &= |u_0^2u_1^{\alpha} - v_0^2v_1^{\alpha}|\le C(|u_0-v_0| + |u_1^{\alpha}-v_1^{\alpha}|).
\end{align*}
This shows Hypothesis (H5)' with $\gamma=\alpha$ (and $\gamma=1$ if $\alpha=0$). Thus, Theorem \ref{thm.wsu} can be applied.

%%%%%%%%%%%%%%

\subsection{Multiphase model for biological tissues}

Multiphase models describe interactions between different components of a cellular fluid, for instance, various cell types, extracellular matrix, interstitial fluid, etc. The evolution of the volume fractions $u_i$ of phase $i$ is governed by
\begin{align*}
  \pa_t u_i = \diver\bigg(\na(u_iq_i(u)) 
  - u_i\sum_{j=1}^n\na(u_jq_j(u))\bigg), \quad i=1,\ldots,n.
\end{align*}
This system was formally derived from fluid equations in the zero-inertia limit in \cite[Prop.~2]{Jue16}. The functions $q_i(u)$ model the intraphase pressures. If $n=1$ and $q_1(u_1)=\frac12 u_1$, we recover the diffusion equation from Section \ref{sec.scalar} with $\alpha=1$. We consider only the case $n=2$ and 
\begin{align*}
  q_1(u) = q_{11}u_1 + q_{12}u_1u_2, \quad 
  q_2(u) = q_{12}u_1u_2 + q_{22}u_2,
\end{align*}
where $q_{ij}$ are positive numbers satisfying $16q_{11}q_{22}>q_{12}^2$. Then the diffusion matrix becomes%
\begin{align*}
  A(u) = \begin{pmatrix}
  u_1(2q_{11} + q_{12}u_2(2-u_2) - 2q_1(u)) & 
  u_1(q_{12}u_1(1-u_1) - 2q_{2}(u)) \\
  u_2(q_{12}u_2(1-u_2) - 2q_{1}(u)) & 
  u_2(2q_{22} + q_{12}u_1(2-u_1) - 2q_2(u))
  \end{pmatrix}.
\end{align*}
We compute, with the entropy density \eqref{1.h},
\begin{align*}
  h''(u)A(u) = \begin{pmatrix}
  2q_{11} + 2q_{12}u_2 & q_{12}u_1 \\
  q_{12}u_2 & 2q_{22} + 2q_{12}u_1
  \end{pmatrix}.
\end{align*}
The symmetric part of this matrix is positive definite, since its determinant can be estimated (in an non-optimal way) as
\begin{align*}
  \det&\bigg(\frac12\big((h''(u)A(u))^T+h''(u)A(u)\big)\bigg) \\
  &= 4q_{11}q_{22} + 4q_{12}(q_{11}u_1+q_{22}u_2)
  + \frac72 q_{12}^2u_1u_2 
  - \frac14q_{12}^2(u_1^2+u_2^2) \\
  &\ge 4q_{11}q_{22} - \frac14 q_{12}^2 > 0.
\end{align*}
Thus, Hypotheses (H3) and (H4) are fulfilled with $s=1$. In particular, the conditions of the boundedness-by-entropy method of \cite{Jue15} are satisfied, and we conclude the existence of a global bounded weak solution $u$. The mobility matrix $B(u)=A(u)h''(u)^{-1}$ has the property $B_{ij}(u)=u_iu_j\beta_{ij}(u)$ for $i,j=1,2$, where $\beta_{ij}(u)$ are polynomials of order at most three. Moreover, a computation shows that 
\begin{align*}
  \bar{B}_{0j}(\bar{u}) = -u_0u_j\beta_{0j}(u), \quad
  \bar{B}_{j0}(\bar{u}) = -u_0u_j\beta_{j0}(u), \quad
  \bar{B}_{00}(\bar{u}) = u_0^2\beta_{00}(u),
\end{align*}
where $j=1,2$ and $\beta_{ij}$ are as before polynomials of order at most three. This yields
\begin{align*}
  \bigg|\frac{\bar{B}_{ij}(\bar{u})}{u_i}
  - \frac{\bar{B}_{ij}(\bar{v})}{v_i}\bigg|
  = |u_j\beta_{ij}(u)-v_j\beta_{ij}(v)| 
  \le C\sum_{i=0}^2|u_i-v_i| 
  \quad\mbox{for }\bar{u},\bar{v}\in\simplex_0,
\end{align*}
where the constant $C>0$ depends on $(q_{ij})$. This shows Hypothesis (H5). 

%%%%%%%%%%%%%%

\subsection{Tumor-growth model}

Jackson and Byrne have derived a tumor-growth model, which describes the dynamics of the volume fractions for the tumor cells $u_1$ and the extra-cellular matrix $u_2$ \cite{JaBy02}. The cross-diffusion system equals \eqref{1.eq} with the diffusion matrix
\begin{align*}
  A(u) = \begin{pmatrix}
  2u_1(1-u_1) - \beta\theta u_1u_2^2 & -2\beta u_1u_2(1+\theta u_1) \\
  -2u_1u_2 + \beta\theta(1-u_2)u_2^2 & 2\beta u_2(1-u_2)(1+\theta u_1)
  \end{pmatrix},
\end{align*}
where $\beta>0$ and $\theta>0$ are parameters in the partial pressures. The variable $u_0=1-u_1-u_2$ is the volume fraction of the interstitial fluid, which is primarily composed of water. In fact, this model follows from the multiphase approach with $q_1(u)=u_1$ and $q_2(u)=\beta u_2(1+\theta u_1)$, but compared to the previous subsection, the pressure functions satisfy the ``non-symmetry'' property $\pa^2 q_1/(\pa u_1\pa u_2)\neq \pa q_2/(\pa u_1\pa u_2)$. The global existence of a weak solution was proved in \cite{JuSt12} assuming that $\theta<4/\sqrt{\beta}$. Under this condition, it is shown in \cite[Sec.~3]{JuSt12} that Hypothesis (H3) is satisfied with $s=1$. We compute
\begin{align*}
  h''(u)A(u) = \begin{pmatrix}
  2 & 0 \\ \beta\theta u_2 & 2\beta(1+\theta u_1)
  \end{pmatrix}.
\end{align*}
We infer that Hypothesis (H4) is satisfied. To verify Hypothesis (H5), we need to compute the entries of the matrix $\bar{B}(\bar{u})$, defined in \eqref{1.barB}. A calculation shows that $\bar{B}_{ij}(\bar{u})=u_iu_j\beta_{ij}(u)$ for $i,j=0,1$, where $\beta_{ij}$ is a polynomial of order at most three. Hence, as in the previous subsection, Hypothesis (H5) is satisfied, and weak--strong uniqueness follows.

%%%%%%%%%%%%%%%%

\subsection{Modified Busenberg--Travis model}

The model \cite{BuTr83} $\pa_t u_i=\diver(u_i\na p_i(u))$ with $p_i(u)=\sum_{j=1}^n p_{ij}u_j$ models the evolution of segregating population species. Assuming that $(p_{ij})$ is symmetric positive definite, the existence of global weak solutions was shown in \cite{JPZ22}, and the boundedness of the weak solutions was proved in \cite{LaMa23} for two species. It is not known whether the general $n$-species model possesses {\em bounded} weak solutions. To achieve such solutions, we modify the equations:
\begin{align*}
  \pa_t u_i = \diver\bigg(u_i\na p_i(u) 
  - \delta u_i\sum_{j=1}^n u_j\na p_j(u)\bigg), \quad i=1,\ldots,n.
\end{align*}
If $\delta=0$, we recover the Busenberg--Travis system (which, strictly speaking, was suggested for $(p_{ij})$ having rank one). If $\delta=1$, the equations can be derived from a fluid model in a multiphase approach \cite{JRX25}. We argue below that the solutions are bounded. In the following, let $\delta=1$. The diffusion matrix is given by its entries
\begin{align*}
  A_{ij}(u) = u_i(p_{ij} - p_j(u)), \quad i,j=1,\ldots,n.
\end{align*}
With the entropy density as in \eqref{1.h}, we obtain $(h''(u)A(u))_{ij}=p_{ij}$. Hence, $h''(u)A(u)$ is positive definite by assumption. Thus, Hypotheses (H3) and (H4) hold with $s=1$. In particular, we can apply the boundedness-by-entropy method of \cite{Jue16} and obtain the existence of a global bounded weak solution. A computation shows that the mobility matrix can be factorized according to $B_{ij}(u)=u_i\beta_{ij}(u)$ for some polynomial $\beta_{ij}$ with $i,j=1,\ldots,n$. Moreover, the augmented mobility matrix satisfies 
\begin{align*}
  \bar{B}_{0j}(\bar{u}) = -u_0u_j\beta_{0j}(u), \quad
  \bar{B}_{j0}(\bar{u}) = -u_0u_j\beta_{j0}(u), \quad
  \bar{B}_{00}(\bar{u}) = u_0^2\beta_{00}(u),
\end{align*}
where $j=1,\ldots,n$ and $\beta_{ij}$ are polynomials. This shows that Hypothesis (H5) is fulfilled.

%%%%%%%%%%%%%%%%

\subsection{Two-species Maxwell--Stefan system}

The Maxwell--Stefan equations describe the diffusive transport of the components of gaseous mixtures. The model consists of the mass and force balance equations
\begin{align*}
  \pa_t u_i + \diver J_i = 0, \quad 
  \na u_i = -\sum_{j=0}^n d_{ij}(u_jJ_i-u_iJ_j), \quad i=0,\ldots,n,
\end{align*}
where $u_i$ are the volume fractions of the gas components, $J_i$ are the partial fluxes, and $d_{ij}$ are symmetric diffusion coefficients. For illustration, we consider only the case $n=2$. Indeed, in this case, th inversion of the gradient--flux relations yields rather easy expressions, and we end up with equations \eqref{1.eq} and the diffusion matrix
\begin{align*}
  A(u) = \frac{1}{a(u)}\begin{pmatrix}
  d_{02} + (d_{12} - d_{02})u_1 & (d_{12} - d_{01})u_1 \\
  (d_{12} - d_{02})u_2 & d_{01} + (d_{12} -d_{01})u_2 \\
  \end{pmatrix},
\end{align*}
where $a(u) = d_{01}d_{02}u_0+d_{01}d_{12}u_1+d_{02}d_{12}u_2$ is strictly positive for $u\in\simplex$. It is proved in \cite[Lemma 3.2]{JuSt12} that Hypothesis (H3) is satisfied with $s=1/2$. The entries of the matrix $B(u)=A(u)h''(u)^{-1}$ satisfy
\begin{align*}
  \bar{B}_{ii}(\bar{u}) = \frac{u_i}{a(u)}\beta_{ii}(u)
  \quad\mbox{for }i=0,1,2, \quad
  \bar{B}_{ij}(\bar{u}) = \frac{u_iu_j}{a(u)}\beta_{ij}(u)
  \quad\mbox{for }i\neq j,
\end{align*}
and $\beta_{ij}$ are affine or quadratic polynomials. We conclude that Hypotheses (H4) and (H5) are fulfilled. Hence, the weak--strong uniqueness property of Theorem \ref{thm.wsu} holds. Notice that this property was already proved in \cite{HJT22}, even for the $n$-species model. 

%%%%%%%%%%%%%%%%%

\subsection{Thin-film solar-cell model}

Thin films for solar cells can be fabricated by physical vapor deposition, which is a process in which material is vaporized in a vacuum and then condensed onto a substrate to form a thin film. The diffusion of the material components can be described by the equations
\begin{align*}
  \pa_t u_i = \diver\bigg(\sum_{j=0}^n a_{ij}(u_j\na u_i-u_i\na u_j)
  \bigg), \quad i=0,\ldots,n,
\end{align*}
where $u_i$ are the volume fractions of the components and $a_{ij}=a_{ji}$. The model was derived in \cite{BaEh18} from a lattice hopping model, and the existence of a global bounded weak solution was proved. The diffusion matrix has the entries
\begin{align*}
  A_{ii}(u) &= \sum_{k=1,\,k\neq i}^n(a_{ik}-a_{i0})u_k + a_{i0}
  \quad\mbox{for }i=1,\ldots,n, \\
  A_{ij}(u) &= -(a_{ij}-a_{i0})u_i\quad\mbox{for }i\neq j.
\end{align*}
It is shown in \cite[Lemma 2.3]{BaEh18} that Hypothesis (H3) is satisfied with $s=1/2$. We compute the entries of the augmented mobility matrix:
\begin{align*}
  \bar{B}_{ii}(\bar{u}) 
  &= u_i\bigg(\sum_{k=1,\,k\neq i}^n a_{ik}u_k + a_{i0}u_0\bigg) \ \mbox{for }i=1,\ldots,n,  \\
  \bar{B}_{ij}(\bar{u}) &= -a_{ij}u_iu_j\quad\mbox{for }i\neq j,\  
  i,j=0,\ldots,n, \quad
  \bar{B}_{00}(\bar{u}) = u_0\sum_{k=1}^n a_{k0}u_i.
\end{align*}
We infer that $\bar{B}_{ij}(\bar{u})/\sqrt{u_iu_j}$ is bounded and $\bar{u}\mapsto\bar{B}_{ij}(\bar{u})/u_i$ is Lipschitz continuous. Hence, Hypotheses (H4) and (H5) are satisfied and the weak--strong uniqueness property holds. This feature was already proved in \cite{HoBu22}. 

%%%%%%%%%%%%%%%%%%%%%%%%%%%%%%%%%%%%%%%%%%%%%%%%%%%%%%%%%%

\section{Extensions}\label{sec.ext}

We show that our approach can be extended to more general situations.

\subsection{Reaction rates}

We claim that Theorem \ref{thm.wsu} still holds when we include suitable reaction terms. We consider 
\begin{align*}
  \pa_t u_i = \diver\bigg(\sum_{j=1}^n A_{ij}(u)\na u_j\bigg)
  + r_i(u), \quad i=1,\ldots,n,
\end{align*}
where we assume that $r_i(u)=0$ if $u_i=0$ and for all $\bar{u}$, $\bar{v}\in\simplex_0$,
\begin{align*}
  \sum_{i=0}^n(r_i(u)-r_i(v)(\log u_i-\log v_i) 
  \le C_R\sum_{i=0}^n\bigg(u_i\log\frac{u_i}{v_i}-u_i+v_i\bigg)
\end{align*}
for some $C_R>0$, where $r_0(u):=-\sum_{i=1}^n r_i(u)$. We obtain as in Section \ref{sec.wsu}:
\begin{align*}
  & H(u(t)|v(t)) - H(u(0)|v(0)) \le I_1 + I_2 + I_3, \quad\mbox{where} \\
  & I_3 = -\sum_{i=0}^n\int_0^t\int_\Omega  \bigg(r_i(u)\log\frac{u_i}{v_i} - r_i(v)\frac{u_i}{v_i}\bigg)dxdx,
\end{align*}
and $I_1$ and $I_2$ are given by \eqref{3.I12}. We set $r_i(u)\log u_i=0$ if $u_i=0$. We only need to estimate $I_3$. For this, we write
\begin{align*}
  I_3 &= -\sum_{i=0}^n\int_0^t\int_\Omega
  \bigg\{(r_i(u)-r_i(v))(\log u_i-\log v_i) 
  + r_i(v)\bigg(\log\frac{u_i}{v_i}-\frac{u_i}{v_i}\bigg)\bigg\}dxds \\
  &\le C_R\sum_{i=0}^n\int_0^t H(u|v)ds
  + \sum_{i=0}^n\int_0^t\int_\Omega
  \frac{r_i(v)}{v_i}\bigg(v_i\log\frac{v_i}{u_i} - v_i + u_i\bigg)dxds,
\end{align*}
where the last step follows from $\sum_{i=0}^n r_i(v)=0$. We compute as in \cite[Lemma 16]{HJT22}:
\begin{align*}
  0 &\le v_i\log\frac{v_i}{u_i} - v_i + u_i 
  = (u_i-v_i)^2\int_0^1\int_0^\theta\frac{dsd\theta}{su_i+(1-s)v_i} \\
  &\le \frac{1}{q}(u_i-v_i)^2\int_0^1\int_0^\theta\frac{dsd\theta}{1-s}
  = C(q)(u_i-v_i)^2,
\end{align*}
where $C(q)>0$ depends on the minimal value $q>0$ of $v_i$. This shows that 
\begin{align*}
  I_3\le C_R\int_0^t H(u|v)ds 
  + C\sum_{i=0}^n\int_0^t\int_\Omega|u_i-v_i|^2 dxds
  \le C\int_0^t H(u|v)ds,
\end{align*}
and we can proceed as in the proof of Theorem \ref{thm.wsu}. 

%%%%%%%%%%%%%%%

\subsection{Ion-channel model}

In this example, Hypotheses (H4) and (H5) are not satisfied. The transport of ions in confined channels can be described by a variant of the Nernst--Planck equations. To simplify the situation, we neglect the electric potential generated by the ions and the background charges. This yields the equations
\begin{align*}
  \pa_t u_i + \diver J_i = 0, \quad
  J_i = -d_i(\na u_i - u_i\na\log u_0), \quad i=1,\ldots,n,
\end{align*}
where the unknowns are the volume fractions $u_1,\ldots,u_n$ of the ions and the volume fraction $u_0$ of the solvent. The flux contains linear diffusion and the Bikerman's excess chemical potential $-\log u_0$ \cite[Sec.~3.1.2]{BKSA09}. The entries of the diffusion matrix are $A_{ij}(u)=d_i(\delta_{ij}-u_i/u_0)$ for $i,j=1,\ldots,n$. It is proved in \cite{JuMa24} that Hypothesis (H3) holds with $s=1/2$ and that there exists a global weak solution satisfying $\na\log u_0\in L^2(\Omega_T)$. In spite of the fact that the weak--strong uniqueness property has been shown in \cite{JuMa24}, we present this example, since we believe that it makes the proof more transparent and understandable.

We compute $B_{ij}(u)=d_iu_i\delta_{ij}$ for $i,j=1,\ldots,n$. The augmented mobility matrix has the additional entries
\begin{align*}
  \bar{B}_{0j}(\bar{u}) = -d_ju_j, \quad 
  \bar{B}_{i0}(\bar{u}) = -d_iu_i, \quad
  \bar{B}_{00}(\bar{u}) = \sum_{i=1}^n d_iu_i, \quad i,j=1,\ldots,n.
\end{align*}  
Thus, Hypotheses (H4) and (H5) are {\em not} fulfilled for $i=j=0$. Still, we can apply the relative entropy method (as first done in \cite{JuMa24}). The proof relies crucially of an improved positive definiteness property of the mobility matrix, also yielding an estimate for $\na u_0$. More precisely, the following result, proved in \cite[Lemma 10]{JuMa24}, holds.

\begin{lemma}
The matrix $G(\bar{u})$, defined by $G_{ij}(\bar{u}) = \bar{B}_{ij}(\bar{u})/\sqrt{u_iu_j}$ for $i,j=0,\ldots,n$, satisfies for any $\bar{u}\in\simplex_0$ and $z\in L(\bar{u})$,
\begin{align*}
  z^T G(\bar{u})z \ge c_A\bigg(\sum_{i=1}^n z_i^2 
  + \frac{z_0^2}{u_0}\bigg),
\end{align*}
where $c_A=\min_{i=1,\ldots,n}d_i>0$. 
\end{lemma}

The additional term $z_0^2/u_0$ allows us to estimate the terms $I_1$ and $I_2$, defined in \eqref{3.I12}. We recall that $Y_i=\sqrt{u_i}\na\log(u_i/v_i)$ and $(P_L(\bar{u})Y)_i=Y_i-\sqrt{u_i}(\sqrt{\bar{u}}\cdot Y)$. We obtain from \eqref{3.I1} that
\begin{align*}
  I_1 &= -\sum_{i,j=0}^n\int_0^t\int_\Omega
  G_{ij}(\bar{u})(P_L(\bar{u})Y)_i(P_L(\bar{u})Y)_j dxds \\
  &\le -c_A\int_0^t\int_\Omega\bigg(\sum_{i=1}^n|(P_L(\bar{u})Y)_i|^2
  + \frac{1}{u_0}|(P_L(\bar{u})Y)_0|^2\bigg)dxds.
\end{align*}
Observe that $\na\log u_0\in L^2(\Omega_T)$ implies that $u_0^{-1}|(P_L(\bar{u})Y)_i|^2$ is integrable. For the second term $I_2$, we exploit the particular structure of $\bar{B}(\bar{u})$. Indeed, since
\begin{align*}
  \frac{\bar{B}_{ij}(\bar{u})}{u_i} - \frac{\bar{B}_{ij}(\bar{v})}{v_i}
  = 0 \quad\mbox{for }i=1,\ldots,n,\ j=0,\ldots,n,
\end{align*}
it remains to estimate $\bar{B}_{0j}(\bar{u})$ for $j=0,\ldots,n$:
\begin{align*}
  I_2 = -\sum_{j=0}^n\int_0^t\int_\Omega
  \bigg(\frac{\bar{B}_{0j}(\bar{u})}{u_0} 
  - \frac{\bar{B}_{0j}(\bar{v})}{v_0}\bigg)\na\log v_j
  \cdot(\sqrt{u_0}Y_0)dxds.
\end{align*}
We split the sum into the two parts $j=0$ and $j=1,\ldots,n$ and insert the definition $\bar{B}_{00}(\bar{u}) = -\sum_{j=1}^n\bar{B}_{0j}(\bar{u})$ as well as $\bar{B}_{0j}(\bar{u})=d_ju_j$:
\begin{align*}
  \sum_{j=0}^n&\bigg(\frac{\bar{B}_{0j}(\bar{u})}{u_0}
  - \frac{\bar{B}_{0j}(\bar{v})}{v_0}\bigg)\na\log v_j \\
  &= \bigg(\frac{\bar{B}_{00}(\bar{u})}{u_0} 
  - \frac{\bar{B}_{00}(\bar{v})}{v_0}\bigg)\na\log v_0
  + \sum_{j=1}^n\bigg(\frac{\bar{B}_{0j}(\bar{u})}{u_0} 
  - \frac{\bar{B}_{0j}(\bar{v})}{v_0}\bigg)\na\log v_j \\
  &= \sum_{j=1}^n\bigg(\frac{\bar{B}_{0j}(\bar{u})}{u_0} 
  - \frac{\bar{B}_{0j}(\bar{v})}{v_0}\bigg)\na\log\frac{v_j}{v_0}
  = \sum_{j=1}^n d_j\bigg(\frac{u_j}{u_0} - \frac{v_j}{v_0}\bigg)
  \na\log\frac{v_j}{v_0}.
\end{align*}
It follows from
\begin{align*}
  \frac{u_j}{u_0} - \frac{v_j}{v_0}
  = \frac{1}{u_0}\bigg((u_j-v_j) - (u_0-v_0)\frac{v_j}{v_0}\bigg)
\end{align*}
that 
\begin{align*}
  I_2 = -\sum_{j=0}^n\int_0^t\int_\Omega d_j
  \bigg((u_j-v_j) - (u_0-v_0)\frac{v_j}{v_0}\bigg)
  \na\log\frac{v_j}{v_0}\cdot\frac{Y_0}{\sqrt{u_0}}dxds.
\end{align*}
As in the proof of Theorem \ref{thm.wsu}, we split $Y_0=(P_L(\bar{u})Y)_0 + (P_{L^\perp}(\bar{u})Y)_0$, i.e.\ $I_2=I_{21}+I_{22}$, where
\begin{align*}
  I_{21} &= -\sum_{j=0}^n\int_0^t\int_\Omega d_j
  \bigg((u_j-v_j) - (u_0-v_0)\frac{v_j}{v_0}\bigg)
  \na\log\frac{v_j}{v_0}\cdot\frac{(P_L(\bar{u})Y)_0}{\sqrt{u_0}}dxds, \\
  I_{22} &= -\sum_{j=0}^n\int_0^t\int_\Omega d_j
  \bigg((u_j-v_j) - (u_0-v_0)\frac{v_j}{v_0}\bigg)
  \na\log\frac{v_j}{v_0}\cdot
  \frac{(P_{L^\perp}(\bar{u})Y)_0}{\sqrt{u_0}}dxds.
\end{align*}
The term $I_{22}$ is estimated as in \eqref{3.I22}:
\begin{align*}
  I_{22} \le C\sum_{i=0}^n\int_0^t\int_\Omega|u_i-v_i|^2 dxds,
\end{align*}
where $C>0$ depends on the $L^\infty(\Omega_T)$ norm of $\na\log v_j$. We apply Young's inequality to the term $I_{21}$:
\begin{align*}
  I_{21} \le \frac{c_A}{2}\sum_{j=0}^n\int_0^t\int_\Omega\frac{1}{u_0}
  |(P_L(\bar{u})Y)_0|^2 dxds
  + C\sum_{j=0}^n\int_0^t\int_\Omega|u_j-v_j|^2 dxds,
\end{align*}
and the first term on the right-hand side is absorbed by $I_1$. We conclude that
\begin{align*}
  H(u(t)|v(t)) &- H(u(0)|v(0)) = I_1 + I_2 \\
  &\le C\sum_{j=0}^n\int_0^t\int_\Omega|u_j-v_j|^2 dxds
  \le C\int_0^t H(u|v)ds,
\end{align*}
and Gronwall's lemma shows that $H(u(t)|v(t))=0$ and hence $u(t)=v(t)$ for $t>0$.

\begin{appendix}
\section{Entropy inequality}\label{sec.app}

We show that the weak solution constructed in \cite[Sec.3]{Jue15} satisfies the entropy inequality \eqref{1.ei} if Hypothesis (H4) holds. 

\begin{lemma}[Entropy inequality]%\label{lem.ei}
There exists a weak solution $u$ to \eqref{1.eq}--\eqref{1.bic} in the sense of Definition \ref{def.weak} fulfilling the entropy inequality \eqref{1.ei}. 
\end{lemma}

\begin{proof}
Let $T>0$, $N\in\N$, let $\tau=T/N$ be a time step size, and $\{t_k=k\tau:k=0,\ldots,N\}$ be a partition of the time intervall $[0,T]$. We consider the approximate problem
\begin{align}\label{a.utau}
  \frac{1}{\tau}\int_0^T\int_\Omega\big(u^{(\tau)}(t)
  - u^{(\tau)}(t-\tau)\big)\cdot\phi dxdt
  &+ \sum_{i,j=1}^n\int_0^T\int_\Omega B_{ij}(u^{(\tau)})
  \na w_j^{(\tau)}\cdot\na\phi_i dxdt \\
  &+ \eps b(w^{(\tau)},\phi) = 0 \nonumber 
\end{align}
for all piecewise constant functions $\phi:(0,T)\to H^m(\Omega;\R^n)$, where $u^{(\tau)}(t)=u^0$ if $-\tau<t<0$, $w_i^{(\tau)}=\log(u_i^{(\tau)}/u_0^{(\tau)})$, $\eps>0$, and
\begin{align*}
  b(w^{(\tau)},\phi) = \int_0^T\int_\Omega\bigg(\sum_{|\alpha|=m}
  \mathrm{D}^\alpha w^{(\tau)}\cdot\mathrm{D}^\alpha\phi  
  + w^{(\tau)}\cdot\phi\bigg)dxdt
\end{align*}
is a regularizing bilinear form. Here, $m\in\N$ is such that $H^m(\Omega)\hookrightarrow L^\infty(\Omega)$,  $\alpha=(\alpha_1,\ldots,\alpha_d)\in\N^d$ is a multi-index, and $\mathrm{D}^\alpha$ denotes the partial derivative $\pa^{|\alpha|}/(\pa x_1^{\alpha_1}\cdots\pa x_d^{\alpha_d})$. Under Hypotheses (H1)--(H3), it is proved in \cite[Sec.~3]{Jue15} that there exists a weak solution $u^{(\tau)}\in L^2(0,T;H^m(\Omega;\R^n))$ to \eqref{a.utau}, which is piecewise constant in time (i.e., $u_i^{(\tau)}$ is constant in time in every subinterval $(t_{k-1},t_k]$), $u^{(\tau)}(x,t)\in\simplex$ for a.e.\ $(x,t)\in\Omega_T$, and the approximate entropy inequality
\begin{align}\label{a.ei}
  \int_\Omega h(u^{(\tau)}(t))dx 
  &+ \sum_{i,j=1}^n\int_0^t\int_\Omega B_{ij}(u^{(\tau)})
  \na w_i^{(\tau)}\cdot\na w_j^{(\tau)} dxds \\
  &+ \eps b(w^{(\tau)},w^{(\tau)}) \le \int_\Omega h(u^0)dx \nonumber 
\end{align}
holds. The resulting uniform bounds and the time-discrete version of the Aubin--Lions compactness lemma yield the following convergences as $(\eps,\tau)\to 0$:
\begin{align*}
  u_i^{(\tau)} \to u_i &\quad\mbox{strongly in }L^2(\Omega_T), \\
  \na(u_i^{(\tau)})^s\rightharpoonup\na u_i^s
  &\quad\mbox{weakly in }L^2(\Omega_T), \\
  \na u_i^{(\tau)}\rightharpoonup\na u_i
  &\quad\mbox{weakly in }L^2(\Omega_T).
\end{align*}
The limit inferior $(\eps,\tau)\to 0$ in \eqref{a.ei} then yields
\begin{align*}%\label{a.ei2}
  \int_\Omega h(u(t))dx + \liminf_{\tau\to 0}
  \sum_{i,j=1}^n\int_0^t\int_\Omega B_{ij}(u^{(\tau)})
  \na w_i^{(\tau)}\cdot\na w_j^{(\tau)} dxds
  \le \int_\Omega h(u^0)dx.
\end{align*}
It remains to estimate the limit inferior. We claim that 
\begin{align*}
  \liminf_{\tau\to 0}
  \sum_{i,j=1}^n\int_0^t\int_\Omega B_{ij}(u^{(\tau)})
  \na w_i^{(\tau)}\cdot\na w_j^{(\tau)} dxds
  \ge \sum_{i,j=1}^n\int_0^t\int_\Omega D_{ij}(u)
  \na u_i\cdot\na u_j dxds,
\end{align*}
where $D(u):=h''(u)A(u)$. To prove this result, we use $\na w_i^{(\tau)} = h''(w_i^{(\tau)})\nabla u_i^{(\tau)}$ and the notation ``:'' for the Frobenius matrix norm,
\begin{align*}
  \sum_{i,j=1}^n & B_{ij}(u^{(\tau)})
  \na w_i^{(\tau)}\cdot\na w_j^{(\tau)} 
  = \na h'(u^{(\tau)})^T:A(u^{(\tau)})
  h''(u^{(\tau)})^{-1}\na h'(u^{(\tau)}) \\
  &= \na (u^{(\tau)})^T:h''(u^{(\tau)})A(u^{(\tau)})\na u^{(\tau)}
  = \sum_{i,j=1}^n D_{ij}(u^{(\tau)})\na u_i^{(\tau)}
  \cdot\na u_j^{(\tau)}.
\end{align*}
Since $(D_{ij}(u^{(\tau)}))$ is positive semidefinite in the sense of Hypothesis (H3), the functional
\begin{align*}
  \mathcal{F}(u^{(\tau)}) = \sum_{i,j=1}^n\int_0^t\int_\Omega
   D_{ij}(u^{(\tau)})\na u_i^{(\tau)}\cdot\na u_j^{(\tau)}dxds
\end{align*}
is bounded from below and convex on $L^2(0,T;H^1(\Omega))$. Therefore, we can apply Tonelli's theorem in the version of \cite[Sec.~8.2]{Eva10} and conclude that $\mathcal{F}$ is weakly lower semicontinuous. This yields
\begin{align*}
  \liminf_{\tau\to 0}\mathcal{F}(u^{(\tau)})\ge \mathcal{F}(u)
\end{align*}
and consequently, by definition \eqref{1.barB} of $\bar{B}(\bar{u})$,
\begin{align*}
  \liminf_{\tau\to 0}&
  \sum_{i,j=1}^n\int_0^t\int_\Omega B_{ij}(u^{(\tau)})
  \na w_i^{(\tau)}\cdot\na w_j^{(\tau)} dxds \\
  &\ge \sum_{i,j=1}^n\int_0^t\int_\Omega B_{ij}(u)
  \na\log\frac{u_i}{u_0}\cdot\na\log\frac{u_j}{u_0} dxds \\
  &= \sum_{i,j=0}^n\int_0^t\int_\Omega\bar{B}_{ij}(\bar{u})
  \na\log u_i\cdot\na\log u_j dxds.
\end{align*}
This shows the entropy inequality \eqref{1.ei}. 
\end{proof} 

\end{appendix}

%%%%%%%%%%%%%%%%%%%%%%%%%%%%%%%%%%%%%%%%%%%%%%%%%%%%%%%%%%%%%%%

\end{document}